\documentclass{article}
\usepackage{graphicx} 
\usepackage{mathrsfs}
\usepackage{mathtools}
\usepackage{amssymb}
\usepackage{parskip}
\usepackage{amsthm}
\usepackage{amsmath}
\usepackage{enumitem}
\usepackage{hyperref}
\usepackage{cleveref}
\usepackage{sectsty}

\usepackage[
backend=biber,
style=alphabetic,
]{biblatex}
\newtheorem{theorem}{Theorem}[section]
\newtheorem{lemma}[theorem]{Lemma}
\newtheorem{definition}[theorem]{Definition}
\newtheorem{remark}[theorem]{Remark}
\newtheorem{corollary}[theorem]{Corollary}

\newcommand{\at}[2][]{#1|_{#2}}
\newcommand{\dom}[1]{\mathcal{D}\left(#1\right)}
\newcommand{\Gammamin}{\Gamma_m}
\newcommand{\xmin}{x_m}

\DeclareMathOperator*{\esssup}{ess\,sup}

\begin{document}

\begin{titlepage}
   \begin{center}
   \uppercase{
       \vspace*{1cm}
       
        \Large
       \textbf{On the properties of the semigroup generated by the RL fractional integral}

       \vspace{0.5cm}

       \large

       A research report submitted to the Scientific Committee of the Hang Lung Mathematics Award
            
       \vspace{1.5cm}

       \normalsize

       \textbf{Author Name}
       
       Ethan Jon Yi Soh and Kyan Ka Hin Cheung
       
       \vspace{1.0cm}

       \textbf{Team number}

       2340158
       
       \vspace{1.0cm}

       \textbf{Teacher}
       
       Mr. Thomas Johnson
       
       \vspace{1.0cm}

       \textbf{School}
       
       Harrow International School Hong Kong
       
       \vspace{1.0cm}

       August 2023
       
    }
   \end{center}
\end{titlepage}

\begin{abstract}
    For operators $A$, it is sometimes possible to define $e^{At}$ as an operator in and of itself provided it meets certain regularity conditions. Like $e^{\lambda x}$ for ODEs, this operator is useful for solving PDEs involving the operator A. We call the set of $e^{At}$ a semigroup generated by $A$. In this paper, we discuss the properties of semigroups generated by the fractional integral, an operator appearing in PDEs in increasingly many fields, over Bochner-Lebesgue spaces.
\end{abstract}

\tableofcontents

\section{Introduction}

The fractional integral is an extension to the ordinary integrals to a non-integer order and has numerous applications in modelling various phenomena such as viscoelasticity, fractionally-damped systems, and diffusion. \cite{kilbas2006} A prominent definition for the fractional integral is the Riemann-Liouville integral which can be derived from the Grünwald-Letnikov fractional derivative or the Cauchy formula for repeated integration and is defined, for order $\alpha$, by:

$$_{x_0} J^\alpha_x f(x) = \frac{1}{\Gamma(\alpha)}\int^{x}_{x_0}f(t)(x-t)^{\alpha-1}dt$$

where $f$ is a function which maps the interval $[x_0, x_1]$ to a Banach Space $X$ (which can be the real numbers, Euclidean vectors or even $L^p$ functions), and $x_0\leq x\leq x_1$. \cite[Definition 5]{carvalhoneto2021}

Semigroups are mappings from positive reals to $\mathscr{L}(F)$, the set of all continuous linear operators in $F \rightarrow F$. For a given one-parameter semigroup $T$, they satisfy the properties $T(t+s) = T(t)T(s)$ and $T(0) = I$, where $I$ is the identity operator on $F$. \cite[Definition 2.1]{isem15} These two properties can allow one to reveal a lot of information about the given semigroup $T(t)$. The semigroups can be further categorised into different types of semigroups, such as $C_0$-semigroups and analytic semigroups, depending on the properties they possess. The infinitesimal generator, $A$, of a one-parameter semigroup, $T$,  is defined to have domain $\dom{A} \coloneqq \{f\in F \mid T(\cdot)f \text{ is differentiable in} \left[0,\infty\right) \}$. Furthermore, if $f\in \dom{A}$, then:

$$Af\coloneqq \frac{\mathrm{d}}{\mathrm{d}t}T(t)f\at[\bigg]{t=0} = \lim_{h\to 0^+}\frac{T(h)f - f}{h}$$

This property is highly useful, as by treating a function $f(x, t)$ as a time-varying vector $f(t)$, it allows one to investigate partial fractional differential equations through the lens of $C_0$-semigroups. \cite{isem15} For example, for an operator $A$, the solution to the abstract Cauchy problem

$$
\left\{\begin{matrix}\dot{f}(t)=Af(t) \\ f(0)=0\end{matrix}\right.
$$

Is known to be

$$f(t)=T(t)f(0)$$

where $T$ is the semigroup generated by $A$.
It is also known that a linear operator is the infinitesimal generator of a uniformly continuous semigroup, a type of one-parameter semigroup, if and only if the operator is also bounded. Normally, derivatives appear in a PDE, but to ensure the well-behavedness of the semigroup generated, we instead use the fractional integral operator, $_{x_0}J^\alpha_x$, as the infinitesimal generator of a unique one-parameter semigroup, which we intend to study in this paper.

In this paper, we will use the theory of one-parameter semigroups in order to separately determine the properties of the semigroup generated by the fractional integral. We use the theory on Bochner-Lebesgue spaces $L^p(x_0, x_1;X)$ to investigate semigroups generated by Riemann-Liouville fractional integrals, namely its boundedness and well-behavedness.

\newpage

\section{Preliminaries of Bochner-Lebesgue Spaces and \texorpdfstring{$\mathbf{C_0}$}{C0}-semigroups}

In this section, we will present some well known definitions and results regarding the classical Bochner-Lebesgue spaces, one-parameter semigroups and the Riemann-Liouville fractional integral, which will be of utter importance throughout our study and analysis of the semigroup generated by the fractional integral. 

\begin{definition} \label{bochnerdef} \emph{\cite[Definition 1]{carvalhoneto2021}} 
Let $E$ be a subspace in $\mathbb{R}^n$, $M$ be a $\sigma$-algebra and $\mu$ be a measure in $(E,\Sigma)$. The representation $(E, M, \mu)$ is called a \emph{measure space}. Let $X$ be a Banach space over $\mathbb{C}$. Then:

\begin{enumerate}[label=(\roman*)]
    \item A step function $\varphi: E\rightarrow X$, where $X$ is an arbitrary Banach space, is \emph{Bochner measurable} if $\varphi^{-1}(\{s\})\in M,\hspace{1mm} \forall s\in X$. Furthermore, if $\mu(\varphi^{-1}(\{s\}))<\infty$, then the function is also \emph{integrable in $E$}.
    \item A Bochner measurable and integrable step function $\varphi : E \rightarrow X$ is a \emph{simple function} if and only if it can be expressed as a summation\footnote{Where $\chi_{A_j}$ is the indicator function of the set $A_j$, $\{a_j\} \text{ is such that }\forall j \hspace{2mm}a_j\in X$ and $A_j$ is chosen such that $\forall j, \hspace{1mm} A_j\subset E;\hspace{1mm} \forall i\neq j,\hspace{1mm} A_i\cap A_j = \phi \text{ and } \bigcup_{j=1}^n A_j = E.$},
$$\varphi = \sum_{j=0}^{n} a_j \chi_{A_j}$$ 
\textit{and its integral is defined as,}
$$\int_E \varphi\hspace{1mm} d\mu = \sum_{j=1}^n a_j\mu(A_j)$$
    \item \textit{A function $f : E \rightarrow X$ is \emph{Bochner measurable} if there exists a sequence $\{\varphi_n(x)\}_{n=1}^\infty$ of simple functions such that $\varphi_n (x)\rightarrow f(x) \text{ as } n\rightarrow \infty \text{ in the topology of } X$, for almost every $x\in E$.}\par

    \item \textit{A function $f : E \rightarrow X$ is  \emph{Bochner integrable} if there exists a sequence $\{\varphi_n(x)\}_{n=1}^\infty$ of simple functions such that,}
$$\lim_{n\rightarrow \infty}\int_E \|\varphi_n(x) - f(x)\|_X d\mu = 0.$$
\end{enumerate}

\end{definition}

\begin{lemma}\label{bochnerprops} 
We have the following two results from these definitions:
\begin{enumerate}[label=(\roman*)]
    \item Consider $I \subset \mathbb{R}$. If $f: I \rightarrow X$ is a Bochner measurable function and $g: \mathbb{R} \rightarrow \mathbb{R}$ is a Lebesgue measurable function, then we their convolution $\mathbb{R} \times I \ni(t, s) \mapsto g(t-s) f(s)$ is Bochner measurable. 

     \item A function $f: I \rightarrow X$ is Bochner integrable if, and only if, $f$ is Bochner measurable and $\|f\|_X \in$ $L^1(I ; \mu)$. This allows us to introduce the concept of \emph{Bochner-Lebesgue spaces}, as defined below.
\end{enumerate}
\end{lemma}

\begin{definition} \label{functionnorm} \emph{\cite[Definition 3]{carvalhoneto2021}}
Consider $1\leq p\leq \infty$. $L^p(I;X)$ denotes the space of all Bochner measurable functions $f : I\rightarrow X$ in which $\|f\|_X\in L^p(I;\mathbb{R})$.\footnote{$L^p(I;\mathbb{R})$ represents the classical Lebesgue space.} $L^p(I;X)$ is a \emph{Banach space} with the norm,
$$\|f\|_{L^p(I;X)} = \begin{cases}
\left[\int_I \|f(s)\|_X^p ds\right]^\frac{1}{p},   \hspace{2.35mm}\text{ if } 1\leq p<\infty\\
\esssup_{s\in I}\|f(s)\|_X,   \text{ if } p = \infty
\end{cases}
$$
\end{definition}

From this, we can define an \emph{operator norm}:

\begin{definition} \label{operatornorm} \emph{\cite{carvalhoneto2021}}
Consider a linear operator $A\in\mathscr{L}(L^p(x_0,x_1;X))$. Then, we define its \emph{operator induced norm} as

$$\|A\|_{\mathscr{L}(L^p(x_0,x_1;X))}\coloneqq\sup_{f\in L^p(x_0,x_1;X)}\frac{\|Af\|_{L^p(x_0,x_1;X)}}{\|f\|_{L^p(x_0,x_1;X)}}$$

Because the notation is cumbersome, we will denote the operator norm of bounded linear operators as simply $\|\cdot\|$, unless there is potential for confusion between different norms.
\end{definition}

Under this norm, $\mathscr{L}(L^p(x_0,x_1;X))$ becomes a Banach space. That being said, this is not the only topology that we can define. We also introduce the \emph{strong operator topology}:

\begin{definition} \label{sot} \emph{\cite{isem15}}
Consider a sequence $(A_n)$ consisting of elements of $\mathscr{L}(L^p(x_0,x_1;X))$. 

\begin{enumerate}[label=(\roman*)]
    \item $(A_n)$ \emph{uniformly converges} to $A$ iff
        $$\lim_{n\to\infty}\|A_n-A\|_{\mathscr{L}(L^p(x_0,x_1;X))}=0$$
    \item $(A_n)$ \emph{strongly converges} to $A$ iff for all $f\in L^p(x_0,x_1;X)$,
        $$\lim_{n\to\infty}\|A_nf-Af\|_{L^p(x_0,x_1;X)}=0$$
\end{enumerate}
The different definitions of convergence implies different notions of open sets, and different topologies on $\mathscr{L}(L^p(x_0,x_1;X))$. Uniform convergence is convergence in the topology induced by the operator norm, while convergence in the \emph{strong operator topology} is precisely strong convergence.

Note that uniform convergence implies strong convergence, but not vice versa.
\end{definition}

We note that many theorems regarding properties of Lebesgue integrals also carry over to Bochner integrals. In particular:

\begin{theorem} \label{fubini} \emph{(Fubini's Theorem for Bochner integrals) \cite[Proposition 1.2.7]{hytonen2016}}
Given $\sigma$-finite measure spaces $S, T$ (that is, the measure spaces are a countable union of elements in their $\sigma$-algebras) and Banach space $X$, if $f: S\times T\to X$ is a function such that 

$$\int_{S\times T}\|f(s, t)\|_X\,d\mu(s, t)<\infty$$

then,

$$\int_{S\times T}f(s, t)\,d\mu(s, t)=\int_S\int_T f(s, t)\,d\mu(t)\,d\mu(s)=\int_T\int_S f(s, t)\,d\mu(s)\,d\mu(t)$$
\end{theorem}

\begin{theorem}\label{leibnizintrule}\emph{\cite[Theorem 53]{carvalhoneto2021} (Leibniz integral rule)} 
Suppose $f: S\times T \rightarrow X$ be a Bochner measurable function, which is Bochner integrable with respect to the second variable, and the functions $\phi$, $\psi: S\rightarrow T$ be differentiable. Suppose further that the partial derivative of $f(s,t)$ exists for almost every $(s,t)\in S\times T$ then the following holds:
$$\frac{\mathrm{d}}{\mathrm{d}s}\int_{\psi(s)}^{\phi(s)}f(s,t)dt = f(s,\phi(s))\phi'(s) - f(s,\psi(s))\psi'(s) +\int_{\psi(s)}^{\phi(s)}\frac{\partial}{\partial s}f(s,t)dt$$
\end{theorem}

We also list some useful inequalities for studying norms:

\begin{corollary}\label{minkowskicorollary}\emph{\cite[Theorem 46]{carvalhoneto2021} (Minkowski's inequality for integrals)}
Suppose that $S$ and $T$ are measure spaces, the function $f: S\times T\rightarrow \mathbb{R}$ is Lebesgue measurable and let $1\leq p\leq\infty$. Then, 
$$\left[\int_T\left|\int_Sf(s,t)ds\right|^pdt\right]^\frac{1}{p}\leq\int_S\left[\int_T|f(s,t)|^pdt\right]^\frac{1}{p}ds$$

\end{corollary}

The following theorem is an immediate consequence of the above (Minkowski's inequality for integrals).

\begin{theorem}\label{C48}\emph{\cite[Theorem 48]{carvalhoneto2021}}
Suppose the function $f: [t_0,t_1] \rightarrow X$ is Bochner integrable and $g:\mathbb{R}\rightarrow [0,\infty)$ is a locally Lebesgue integrable function. If $1\leq p\leq\infty$ and $0\leq h\leq t_1-t_0$ then the following inequality holds:
\begin{multline*}
\left[\int_0^{t_1-t_0-h}\left[\int_0^rg(s)\|f(r+t_0-s)\|_Xds\right]^pdr\right]^\frac{1}{p}\\\leq\int_0^{t_1-t_0-h}g(s)\left[\int_s^{t_1-t_0-h}\|f(r+t_0-s)\|_X^pdr\right]^\frac{1}{p}ds
\end{multline*}
\end{theorem}

\begin{theorem}\label{holder}\emph{(Hölder's inequality)}
Let E be a measure space, $1\leq p, q\leq\infty$ be numbers such that $\frac{1}{p}+\frac{1}{q}=1$\footnote{We define $\frac{1}{\infty}=0$} and functions $f,g: E\to\mathbb{C}$ be measurable functions. Then the following inequality holds:
$$\|fg\|_{L^1(E)}\leq \|f\|_{L^p(E)}\|g\|_{L^q(E)}$$
\end{theorem}

We now give some classical definitions and results regarding the theory of one-parameter semigroups, which is a powerful tool used in functional analysis and will be critical to our study.

\begin{definition} \label{semigroupdef} \emph{\cite[Definition 2.1]{isem15}}
Let $T : \left[0,\infty\right) \rightarrow \mathscr{L}(X)$ be a mapping. Then:

\begin{enumerate}[label=(\roman*)]
    \item \textit{$T$ is said to have the \emph{semigroup property} if, for all $t,s\in \left[0,\infty\right)$,}
$$T(t+s) = T(t)T(s)$$
and\footnote{$I$ is the identity operator on $X$} $$T(0) = I$$

    \item \textit{Suppose the function $T : \left[0,\infty\right) \rightarrow \mathscr{L}(X)$ has the semigroup property. If the mapping:
    $$t\longmapsto T(t)f\in X$$ 
    is continuous $\forall f\in X$, then $T$ is a \emph{strongly continuous one-parameter semigroup}  of bounded linear operators on Y.}\footnote{$T$ can equivalently be called a $C_0$-semigroup on $X$.}
\end{enumerate}
\end{definition}

We also have the following property of Bochner-Lebesgue spaces:

\begin{theorem}\label{smoothdense} \emph{\cite[Lemma 2.5]{neumeister2021curve}}
    Let $I=(x_0, x_1)$. Then, for each $f\in L^p(I; X)$ and $\varepsilon>0$, there exists a function $\phi_\varepsilon\in C^{\infty}(I; X)$ such that $\|f-\phi_\varepsilon\|_{L^p(I; X)}<\varepsilon$.
\end{theorem}

We first introduce a lemma regarding sets of bounded operators:

\begin{lemma} \label{uniformbound} \emph{\cite[Theorem 2.28]{isem15} (Uniform Boundedness Principle)}
Let $X, Y$ be a Banach space and let $S$ be a subset of $\mathscr{L}(X, Y)$. Then, if for all $x\in X$, we find

$$\sup\{\|Ax\|\mid A\in S\}<\infty$$

we say $S$ is \emph{uniformly bounded} - that is,

$$\sup\{\|A\|\mid A\in S\}<\infty$$
\end{lemma}

\begin{theorem} \label{semigroupbound} \emph{\cite[Proposition 2.2]{isem15}}
Let $T : \left[0,\infty\right) \rightarrow \mathscr{L}(X)$ be a $C_0$-semigroup. Then $\forall t\geq 0$,\par
\begin{enumerate}[label=(\roman*)]
    \item T is \emph{locally bounded}, meaning,
$$\sup_{s\in \left[0,t\right]}\|T(s)\|<\infty$$

    \item There exists constants $M\geq 1 \text{ and }\omega \in \mathbb{R}$ such that the following inequality holds:
$$\|T(t)\|\leq Me^{\omega t}$$
    Where the semigroup $T$ is said to be of \emph{type $(M, \omega)$} if it satisfies the above inequality with the particular constants $M$ and $\omega$.
\end{enumerate}
\end{theorem}

\begin{proof}
For a fixed function $f\in X$, $T(\cdot)f$ is continuous on $[0,\infty)$ and, thus, bounded on compact intervals $[0,t]$:
$$\sup_{s\in[0,t]}\|T(s)f\|<\infty$$
Therefore, by \Cref{uniformbound},
$$\implies \sup_{s\in[0,t]}\|T(s)\|<\infty$$
This gives us our first result (i). The second result follows from the first, as we now define:
$$M\coloneqq \sup_{s\in[0,1]}\|T(s)\|<\infty$$
Let $t\geq 0$ be arbitrary and $t = n+r$, where $n\in \mathbb{N}$ and $r\in[0,1)$. This allows us to obtain:
\begin{flalign*}
\|T(t)\|\leq \|T(r)T(1)^n\| \leq M\|T(1)\|^n &\leq M(\|T(1)\|+1)^n\\
                                             &\leq M(\|T(1)\|+1)^t = Me^{\omega t}
\end{flalign*}
 where we set $\omega\coloneqq \ln({\|T(1)\|+1})$, which completes the proof for the second result.
\end{proof}

\begin{definition} \label{generator} \emph{\cite[Definition 2.7]{isem15}}
The infinitesimal generator $A$ of a semigroup $T$ is defined to have domain 

$$\dom{A} \coloneqq \{f\in F \mid T(\cdot)f \text{ is differentiable in} \left[0,\infty\right) \}\subseteq X$$

Furthermore, if $f\in \dom{A}$, then:

$$Af\coloneqq \frac{\mathrm{d}}{\mathrm{d}t}T(t)f\at[\bigg]{t=0} = \lim_{h\to 0^+}\frac{T(h)f - f}{h}$$

\end{definition}

We will now present some important properties of the infinitesimal generator of one-parameter semigroups.

\begin{theorem} \label{generatorprops} \emph{\cite[Proposition 2.9]{isem15}} 
Let $T : [0,\infty)\rightarrow \mathscr{L}(X)$ be a $C_0$-semigroup in $X$ and $A : \dom{A}\rightarrow X$ be its infinitesimal generator. Then:

\begin{enumerate}[label=(\roman*)]
    \item \textit{$A\text{ is a \emph{linear operator} in }\dom{A}$.}
    \item \textit{for $f\in X$}
    $$\int_0^t T(s)f\hspace{1mm}ds \in \dom{A}\text{\hspace{2mm} and \hspace{2mm}} T(t)f-f = A\left(\int_0^t T(s)f\hspace{1mm} ds\right)$$

    \item \textit{For $f\in \dom{A}\text{ then we have that }T(t)f\in \dom{A}$ and,}
    $$\frac{\mathrm{d}}{\mathrm{d}t}T(t)f = AT(t)f = T(t)Af$$
\end{enumerate}
\end{theorem}

We will now introduce a further classification of one-parameter semigroups, known as \emph{uniformly continuous semigroups}, which will be significant to this study.

\begin{theorem} \label{uniformsemigroup} \emph{\cite[Definition 25]{carvalhoneto2021}} 
A uniformly continuous semigroup is a \emph{strongly continuous one-parameter semigroup} $T$ such that:
$$\lim_{t\to0^+}\|T(t)-I\| = 0$$
and can be expressed as,
$$T(t) = e^{At}$$
where, $A$, its infinitesimal generator, is bounded and defined to have a domain $\dom{A} = X$.
Conversely, an operator $A : X\rightarrow X$ is the generator of a uniformly continuous semigroup given by:
$$T(t)\coloneqq e^{At}$$
if and only if $A$ is a \emph{bounded linear operator}.

\end{theorem}

Finally, we show that the fractional integral is a bounded linear operator when mapping from $L^p(x_0,x_1;X)$ to $L^p(x_0,x_1;X)$:

\begin{theorem} \label{fracintbound}\emph{\cite[Theorems 11, 12]{carvalhoneto2021} (Boundedness of Riemann-Liouville integral)}
The Riemann-Liouville fractional integral operator is a bounded linear operator from $L^p(x_0,x_1;X)$ into itself for all $1\leq p\leq \infty$ and its bound is given by:
$$\|_{x_0}J^\alpha_xf\|_{L^p(x_0,x_1;X)}\leq \left[\frac{(x_1 - x_0)^\alpha}{\Gamma(\alpha+1)}\right]\|f\|_{L^p(x_0,x_1;X)}$$
\end{theorem}
\begin{proof}
We will present this proof in the form of several algebraic manipulations. Firstly, we define a dummy variable $s$ such that $s=x-t$, which allows us to obtain:
\begin{flalign*}
\Gamma(\alpha)\|_{x_0}J^\alpha_xf\|_{L^p(x_0,x_1;X)}&=\Gamma(\alpha)\left[\int_{x_0}^{x_1}\|_{x_0}J^\alpha_xf(x)\|_X^pdx\right]^\frac{1}{p}\\
                                                                 &\leq \left[\int_{x_0}^{x_1}\left[\int_{x_0}^{x}(x-t)^{\alpha-1}\|f(t)\|_Xdt\right]^pdx\right]^\frac{1}{p}\\
                                                                 &=\left[\int_{x_0}^{x_1}\left[\int_{0}^{x-x_0}s^{\alpha-1}\|f(x-s)\|_Xds\right]^pdx\right]^\frac{1}{p} 
\end{flalign*}
Now we define $r$ such that $x = r+x_0$ and by applying Theorem \ref{C48} we obtain:

\begin{flalign*}
\Gamma(\alpha)\|_{x_0}J^\alpha_xf\|_{L^p(x_0,x_1;X)}&\leq\left[\int_0^{x_1-x_0}\left[\int_0^rs^{\alpha-1}\|f(r+x_0-s)\|_Xds\right]^pdr\right]^\frac{1}{p}\\
                                                                 &\leq\int_0^{x_1-x_0}\left[\int_0^{x_1-x_0}\|f(r+x_0-s)\|_X^pdr\right]^\frac{1}{p}ds
\end{flalign*}

Finally, by defining $l$ such that $r = l+s-t_0$ we acquire:
\begin{flalign*}
\Gamma(\alpha)\|_{x_0}J^\alpha_xf\|_{L^p(x_0,x_1;X)}&\leq\int_0^{x_1-x_0}s^{\alpha-1}\left[\int_{x_0}^{x_1-s}\|f(l)\|_X^pdl\right]^\frac{1}{p}ds\\
                                                                 &\leq\left[\int_0^{x_1-x_0}s^{\alpha-1}ds\right]\left[\int_{x_0}^{x_1-s}\|f(l)\|_X^pdl\right]^\frac{1}{p}\\
                                                                 &=\left[\frac{(x_1-x_0)^\alpha}{\alpha}\right]\|f\|_{L^p(x_0,x_1;X)}
\end{flalign*}
$$\implies\|_{x_0}J^\alpha_xf\|_{L^p(x_0,x_1;X)}\leq\left[\frac{(x_1 - x_0)^\alpha}{\Gamma(\alpha+1)}\right]\|f\|_{L^p(x_0,x_1;\mathbb{X})}$$

\end{proof}
Thus, by \Cref{uniformsemigroup}, this ensures that the fractional integral is the infinitesimal generator of a unique one-parameter semigroup.

We present a few results regarding Gamma functions from:

\begin{lemma} \label{stirling} \emph{\cite[Equation 5.11.3]{gammasymp} (Stirling's approximation)}
$$\Gamma(z)\sim\sqrt{\frac{2\pi}{z}}\left(\frac{z}{e}\right)^z$$ 
\end{lemma}

\begin{lemma} \label{gammafrac} \emph{\cite[Equation 5.11.12]{gammasymp}}
$$\frac{\Gamma(z+a)}{\Gamma(z)}\sim z^a$$ 
\end{lemma}

\begin{remark} \label{gammamin}
    From numerical computation, we also find that for $x>0$, $\min\{\Gamma(x)|x\in\mathbb{R}^+\}\approx0.885603$ with $x\approx1.46163$. We henceforth define $\Gammamin\coloneqq\min\{\Gamma(x)|x\in\mathbb{R}^+\}$, and $\xmin\coloneqq\Gamma^{-1}(\Gammamin)$.
\end{remark}

We also define the \emph{digamma} function as $\psi(z)\coloneqq \frac{\Gamma'(z)}{\Gamma(z)}$. We have:

\begin{lemma} \label{digammaseries} \emph{\cite[Equation 25]{digamma}}
    For integer $n\in\mathbb{Z}$,
    $$\psi(n)=-\gamma+\sum^{n-1}_{k=1}\frac{1}{k}$$
\end{lemma}

\begin{lemma} \label{digammasymp} \emph{\cite[Equation 16]{digamma}}
    As $x\to\infty$,

    $$\psi(x)\sim\ln x$$
\end{lemma}

Furthermore, by the log-convexity of the log-Gamma function, the digamma function is monotonically increasing.

Finally, we present a function that appears often in the study of fractional Riemann-Liouville integrals:

\begin{definition} \label{mittagleffler}
The \emph{Mittag-Leffler function} $E_{\alpha}$ is defined as

$$E_{\alpha}(t)=\sum^\infty_{k=0}\frac{t^k}{\Gamma(\alpha k+1)}$$

\end{definition}

\begin{remark} \label{mittaglefflerbound}
We can show that for $\alpha>0$, the sequence $\alpha k+1$ is strictly increasing with $\alpha$. Then, by the ratio test, we have

$$\lim_{k\to\infty}\frac{t^{k+1}}{\Gamma(\alpha k+1+\alpha)}\frac{\Gamma(\alpha k+1)}{t^k}=\lim_{k\to\infty}t\frac{\Gamma(\alpha k+1+\alpha)}{\Gamma(\alpha k+1)}$$

Using \Cref{gammafrac}, 

$$\lim_{k\to\infty}t\frac{\Gamma(\alpha k+1+\alpha)}{\Gamma(\alpha k+1)}=\lim_{k\to\infty}t(\alpha k+1)^{-\alpha}=0$$

Hence the Mittag-Leffler function converges for all $\alpha>0$.
\end{remark}

\newpage

\section{The spectrum and resolvent of \texorpdfstring{$_{x_0}J^\alpha_x$}{the alpha-order integral}}

In this section, we explain the definition of the \emph{resolvent} and its applications in $C_0$-semigroups.\footnote{For more details on resolvents, see \cite{smith2015}.} We also explicitly calculate a closed form expression for the resolvent of the fractional integral.

\begin{definition} \label{resolvents} \emph{\cite[Definition 2.22]{isem15}}
Let $A$ be a closed operator defined on the linear subspace $\dom{A}$ of a Banach space $X$. Then, 

\begin{enumerate}[label=(\roman*)]
    \item The \emph{spectrum} of $A$ is the set:
$$\sigma (A) \coloneqq \{\lambda \in \mathbb{C} \mid \lambda I-A : \dom{A}\rightarrow X \text{ is not bijective}\}$$

    \item If $\lambda\not\in\sigma(A)$, then $(\lambda I-A)$ is also injective, meaning that its algebraic inverse $(\lambda I - A) ^{-1}$ exists and is known as the resolvent of set A at point $\lambda$, denoted as:
$$R(\lambda, A) \coloneqq (\lambda I-A)^{-1}$$

    \item The \emph{spectral radius} of $A$ is defined as:
    $$r(A)\coloneqq \max\{|\lambda|\mid\lambda\in\sigma(A)\}$$
\end{enumerate}
\end{definition}

 We present a useful lemma for finding the spectrum:

\begin{lemma} \label{gelfands}
Let $A$ be a bounded operator. Then,

$$r(A)=\lim_{n\to\infty}\|A^n\|^\frac{1}{n}$$
\end{lemma}

We first provide method of computing the resolvent operator of a set:

\begin{theorem} \label{neumann} \emph{(Neumann series)}
Suppose $A$ is an operator acting on Banach space $X$ and $\lambda\in\mathbb{C}$ is a number such that $|\lambda|<r(A)$. Then, the series

$$\frac{1}{\lambda}\sum^\infty_{k=0}\frac{A^k}{\lambda^{k}}$$

converges towards the resolvent operator $R(\lambda, A)$.
\end{theorem}

Finally, we cite an important theorem:

\begin{theorem} \label{hilleyosida} \emph{(Hille-Yosida Theorem) \cite[Theorem 13.37]{rudin1991}}
Let $T$ be a uniformly continuous semigroup of type $(M, \omega)$ with the infinitesimal generator $A$, which is densely defined on $X$. Then, for all $\lambda\in\mathbb{R}$ where $\lambda>\omega$.

    $$\left\|R(\lambda, A)^n\right\|\leq\frac{M}{(\lambda-\omega)^n}$$
\end{theorem}

The proof in its entirety can be found in \Cref{appendix}.

Given all of these results, we can begin studying the spectrum of the operator $J^\alpha$ for $\alpha>0$.

\begin{theorem}
The spectrum of $_{x_0}J^\alpha_x$ is the set $\{0\}$.
\end{theorem}
\begin{proof}
First, we note that $_{x_0}J^\alpha_x$ is not injective. This is because for every $g\in L^p(x_0,x_1;X)$, there exists $f$ such that $_{x_0}J^\alpha_xf=g$ only if $g(x_0)=0$.

Now, using \Cref{gelfands} and \Cref{fracintbound}, we find that

\begin{align*}
r\left(_{x_0}J^\alpha_x\right)=\lim_{n\to\infty}\left\|_{x_0}J^{\alpha n}_{x}\right\|_{L^p(x_0,x_1;X)}^\frac{1}{n}\leq\lim_{n\to\infty}\left[\frac{(x_1-x_0)^{\alpha n}}{\Gamma(\alpha n+1)}\right]^\frac{1}{n}=\lim_{n\to\infty}\frac{(x_1-x_0)^\alpha}{[\Gamma(\alpha n+1)]^\frac{1}{n}}
\end{align*}

Using Stirling's approximation, we find that

\begin{align*}
    r({_{x_0}J^\alpha_{x}})&=\lim_{n\to\infty}\frac{(x_1-x_0)^\alpha}{\left[\sqrt{\frac{2\pi}{\alpha n +1}}\left(\frac{\alpha n +1}{e}\right)^{\alpha n+1}\right]^\frac{1}{n}}\\
    &=\lim_{n\to\infty}(x_1-x_0)^\alpha \left(\frac{2\pi}{\alpha n+1}\right)^{-\frac{1}{2n}}\left(\frac{e}{\alpha n+1}\right)^{\alpha +\frac{1}{n}}\\
    &\leq(x_1-x_0)^\alpha\lim_{n\to\infty}\left(\frac{e}{\alpha n+1}\right)^\alpha=0
\end{align*}

Hence, the spectral radius of $_{x_0}J^\alpha_x$ is $0$ and the series converges.

\end{proof}

\begin{remark}
    Notably, for fractional derivatives, the spectrum is unbounded. For example, for the second derivative on functions with domain $[x_0, x_1]$, the spectrum consists of the points $-\frac{x_1-x_0}{2\pi}n^2$ for $n\in\mathbb{N}$.
\end{remark}

Now we explicitly find the resolvent operator:

\begin{theorem}
For $f\in L^p(I;X)$,

$$R\left(\lambda,\, _{x_0}J^\alpha_x\right)f(x)=\frac{1}{\lambda}\left[f(x)+\int^x_{x_0}\frac{\partial}{\partial x}E_\alpha\left(\lambda^{-1}(x-s)^\alpha\right)f(s)\,ds\right]$$ for $\lambda\neq0$.
\end{theorem}

\begin{proof}
Because $\lambda\neq0$, we can apply \Cref{neumann}:

$$
\begin{aligned}
    _{x_0}J_{x}R\left(\lambda,\, _{x_0}J^\alpha_x\right)f(x)&=\frac{1}{\lambda}\sum^\infty_{k=0}\lambda^{-k}{_{x_0}J^{1+k\alpha}_{x}}f(x) \\
    &=\frac{1}{\lambda}\sum^\infty_{k=0}\int^x_{x_0}\frac{\lambda^{-k}(x-s)^{k\alpha}}{\Gamma(k\alpha+1)}f(s)\,ds \\
    R\left(\lambda,\, _{x_0}J^\alpha_x\right)f(x)&=\frac{1}{\lambda}\frac{d}{dx}\sum^\infty_{k=0}\int^x_{x_0}\frac{\lambda^{-k}(x-s)^{k\alpha}}{\Gamma(k\alpha+1)}f(s)\,ds \\
\end{aligned}
$$

We can consider the summation as integration over $\mathbb{N}$ with the counting measure.

\begin{align*}
    \int_{\mathbb{N}\times[x_0, x]}\left\|\frac{\lambda^{-k}(x-s)^{k\alpha}}{\Gamma(k\alpha+1)}f(s)\right\|_X\,d(s, k)&=\int^x_{x_0}\sum^\infty_{k=0}\frac{\lambda^{-k}(x-s)^{k\alpha}}{\Gamma(k\alpha+1)}\|f(s)\|_X\,ds\\
    &=\int^x_{x_0} E_\alpha\left(\lambda^{-1}(x-s)^\alpha\right)\|f(s)\|_X\,ds \\
    &\leq E_\alpha\left(\lambda^{-1}(x-x_0)^\alpha\right)\int^x_{x_0} \|f(s)\|_X\,ds
\end{align*}

Which is less than $\infty$ by \Cref{mittaglefflerbound}.

Hence, we can apply Fubini's Theorem (\ref{fubini}):

\begin{align*}
    R\left(\lambda,\, _{x_0}J^\alpha_x\right)f(x)&=\frac{1}{\lambda}\frac{d}{dx}\int^x_{x_0}\sum^\infty_{k=0}\frac{\lambda^{-k}(x-s)^{k\alpha}}{\Gamma(k\alpha+1)}f(s)\,ds \\
    &=\frac{1}{\lambda}\frac{d}{dx}\int^x_{x_0}E_\alpha\left(\lambda^{-1}(x-x_0)^\alpha\right)f(s)\,ds \\
    &=\frac{1}{\lambda}\left[f(x)+\int^x_{x_0}\frac{\partial}{\partial x}E_\alpha\left(\lambda^{-1}(x-s)^\alpha\right)f(s)\,ds\right] \\
\end{align*}

The last line resulting from applying \Cref{leibnizintrule}.

\end{proof}

\newpage

\section{Boundedness of the semigroup generated by the fractional integral}

In this section, we will apply \Cref{hilleyosida} to determine an exact bound for the size (operator norm) of the semigroup generated by the fractional integral, which we will define as:
$$\Phi(\alpha,t)\coloneqq e^{_{x_0}J^\alpha_xt} = \sum_{k=0}^\infty{_{x_0}J^{k\alpha}_{x}}\frac{t^k}{k!}$$

\begin{lemma}\label{resolventriangle}
    Let $$Af(x)=\int^x_{x_0}\frac{\partial}{\partial x}E_\alpha\left(\lambda^{-1}(x-s)^\alpha\right)f(s)\,ds$$ Then the semigroup generated by $J^\alpha$ is of type $(1, \omega)$ iff for all $\lambda$ such that $\Re(\lambda)>\omega$, we have

    $$\|A\|\leq\sum^\infty_{k=1}\frac{\omega^k}{\lambda^k}$$
\end{lemma}

\begin{proof}
    
    Suppose $\left\|R\left(\lambda,\, _{x_0}J^\alpha_x\right)\right\|<\frac{1}{\lambda-\omega}$. Then,
    $$\left\|\left(\lambda I-_{x_0}J^\alpha_x\right)^{-n}\right\|\leq\left\|\left(\lambda I-_{x_0}J^\alpha_x\right)^{-1}\right\|^n=\frac{1}{(\lambda-\omega)^n}$$

    Hence, it suffices to show that $\left\|R\left(\lambda,\, _{x_0}J^\alpha_x\right)\right\|<\frac{1}{\lambda-\omega}$.

    Now note that by the triangle inequality, $\left\|R\left(\lambda,\, _{x_0}J^\alpha_x\right)\right\|\leq\frac{1}{\lambda}(\|I\|+\|A\|)=\frac{1}{\lambda}(1+\|A\|)$. Hence, it is sufficient to show that

    \begin{align*}
        \frac{1}{\lambda}(1+\|A\|) &\leq \frac{1}{\lambda-\omega} \\
        1+\|A\| &\leq \frac{\lambda}{\lambda-\omega} \\
        \|A\| &\leq \frac{\lambda}{\lambda-\omega}-1 \\
        &=\frac{\omega}{\lambda-\omega} \\
        &=\sum^\infty_{k=1}\frac{\omega^k}{\lambda^k} \\
    \end{align*}
    
\end{proof}

We note a lemma for finding the norm of an operator:
\begin{lemma}\label{normfind}
    Suppose there exists an operator $S\in\mathscr{L}(L^p(x_0,x_1;X))$ such that for $f\in L^p(x_0,x_1;X)$ we find a.e. that if:

    $$S f(x)=\int^x_{x_0} g(x-s)f(s)\, ds$$

    Then, for all $1\leq p\leq\infty$,

    $$\|S\|_{\mathscr{L}(L^p(x_0,x_1;X))}\leq\int^{x_1-x_0}_0 |g(w)|\, dw$$
\end{lemma}
\begin{proof}
    First, consider the case $1<p<\infty$.
    Set $w=x-s:$

    \begin{align*}
        \left[\int^{x_1}_{x_0}\left\|Sf(x)\right\|^p_X\,dx\right]^{\frac{1}{p}}
        &=\left[\int^{x_1}_{x_0}\left\|\int^{x}_{x_0}g(x-s)f(s)\,ds\right\|^p_X\,dx\right]^{\frac{1}{p}} \\
        &\leq\left[\int^{x_1}_{x_0}\left[\int^{x}_{x_0}|g(x-s)|\|f(s)\|_X\,ds\right]^p\,dx\right]^{\frac{1}{p}} \\
        &=\left[\int^{x_1}_{x_0}\left[\int^{x-x_0}_{0}|g(w)|\|f(x-w)\|_X\,dw\right]^p\,dx\right]^{\frac{1}{p}} \\
    \end{align*}

    Now set $x=r+x_0$:
    \begin{align*}
        \left[\int^{x_1}_{x_0}\left\|Sf(x)\right\|^p_X\,dx\right]^{\frac{1}{p}}\leq\left[\int^{x_1-x_0}_0\left[\int^{r}_{0}|g(w)|\|f(r+x_0-w)\|_X\,dw\right]^p\,dr\right]^{\frac{1}{p}} \\
    \end{align*}

    We then apply \Cref{minkowskicorollary} (as $p\neq\infty$) to the RHS:
    \begin{align*}
        \left[\int^{x_1}_{x_0}\left\|Sf(x)\right\|^p_X\,dx\right]^{\frac{1}{p}}\leq \int^{x_1-x_0}_0|g(w)|\left[\int^{x_1-x_0}_w\|f(r+x_0-w)\|^p_X\,dr\right]^\frac{1}{p}\,dw \\
    \end{align*}

    Notice that $r+x_0-s$ ranges from $x_0$ to $x_1-w$. Hence, $\left[\int^{x_1-x_0}_w\|f(r+x_0-w)\|^p_X\,dr\right]^\frac{1}{p}\leq\left[\int^{x_1}_{x_0}\|f(x)\|^p_X\,dx\right]^\frac{1}{p}$ giving us

    \begin{align*}
        \left[\int^{x_1}_{x_0}\left\|Sf(x)\right\|^p_X\,dx\right]^{\frac{1}{p}}       
        &\leq\int^{x_1-x_0}_0|g(w)|\|f\|_{L^p(x_0,x_1;X)}\,dw \\
        &=\|f\|_{L^p(x_0,x_1;X)}\int^{x_1-x_0}_0|g(w)|\,dw \\
    \end{align*}

    Hence, by definition of the operator norm, we find

    $$\|S\|_{\mathscr{L}(L^p(x_0,x_1;X))}\leq \int^{x_1-x_0}_0|g(w)|\,dw$$

    Now we consider the $L^\infty$ case.

    \begin{align*}
        \left\|Sf(x)\right\|_X&=\left\|\int^x_{x_0}g(x-s)f(s)\,ds\right\|_X\\
        &\leq \int^x_{x_0}g(x-s)\|f(s)\|_X\,ds 
    \end{align*}

    Using \Cref{holder} we find

    \begin{multline*}
        \left\|Af(x)\right\|_X\leq\int^x_{x_0}g(x-s)\|f(s)\|_X\,ds \\
        \leq \left|\int^x_{x_0}g(x-s)\,ds\right|\esssup_{s\in[x_0, x]}\|f(s)\|_X
    \end{multline*}

    Now we apply the substitution $w=x-s$:
    \begin{align*}
        \left|\int^x_{x_0}g(x-s)\,ds\right|&=\left|\int^0_{x-x_0}g(w)\,dw\right| \\
        &=\left|\int^{x-x_0}_0g(w)\,dw\right| \\
        &\leq\int^{x-x_0}_0|g(w)|\,dw
    \end{align*}

    Hence we conclude
    
    \begin{align*}
        \left\|Sf(x)\right\|_X&\leq\int^{x-x_0}_0|g(w)|\,dw\esssup_{s\in[x_0, x]}\|f(s)\|_X \\
        \esssup_{x\in[x_0, x_1]}\left\|Sf(x)\right\|_X&\leq\int^{x-x_0}_0|g(w)|\,dw\esssup_{x\in[x_0, x_1]}\|f(s)\|_X \\
        \|S\|_{\mathscr{L}(L^\infty(x_0,x_1;X))}&\leq \int^{x-x_0}_0|g(w)|\,dw
    \end{align*}

    Which concludes our proof.
    
\end{proof}

We first find the norm of $A$:
\begin{theorem}\label{anorm}
    $$\|A\|_{\mathscr{L}(L^p(x_0,x_1;X))}\leq E_\alpha\left(\lambda^{-1}(x_1-x_0)^\alpha\right)-1$$
    for all $1\leq p\leq \infty$.
\end{theorem}
\begin{proof}

    From \Cref{normfind} we find that

    $$\|A\|_{\mathscr{L}(L^p(x_0,x_1;X))}\leq\int^{x_1-x_0}_0\left|\frac{\partial}{\partial w}E_\alpha\left(\lambda^{-1}w^{\alpha}\right)\right|\,dw$$

    Because $E_\alpha$ is an increasing function, we find

    \begin{align*}
        \|A\|_{\mathscr{L}(L^p(x_0,x_1;X))}&\leq\int^{x_1-x_0}_0\frac{\partial}{\partial w}E_\alpha\left(\lambda^{-1}w^{\alpha}\right)\,dw \\
        &\leq E_\alpha\left(\lambda^{-1}(x_1-x_0)^{\alpha}\right)-E_\alpha(0) \\
        &=E_\alpha\left(\lambda^{-1}(x_1-x_0)^{\alpha}\right)-1 \\
    \end{align*}

\end{proof}

Then we set appropriate bounds on $\omega$ resultingly:
\begin{theorem}
    Let $\Phi$ be the semigroup generated by the operator $_{x_0}J^\alpha_x$. Then, for all $L^p$ spaces where $1\leq p<\infty$, $\Phi$ is a type $(1, \omega)$ semigroup for all
    
    $$\omega\geq\frac{(x_1-x_0)^{\alpha}}{\Gamma(\alpha+1)}$$

    that is,

    $$\|\Phi(\alpha, t)\|_{\mathscr{L}(L^p(x_0,x_1;X))}\leq e^{\frac{(x_1-x_0)^{\alpha}}{\Gamma(\alpha+1)} t}$$
\end{theorem}

\begin{proof}

    We apply the series definition on the norm of $A$ as found above:

    \begin{align*}
        E_\alpha\left(\lambda^{-1}(x_1-x_0)^\alpha\right)-1&=\left|\sum^\infty_{k=1}\frac{\lambda^{-k}(x_1-x_0)^{\alpha k}}{\Gamma(\alpha k+1)}\right| \\
        &=\sum^\infty_{k=1}\frac{\frac{(x_1-x_0)^{\alpha k}}{\Gamma(\alpha k+1)}}{\lambda^k} \\
    \end{align*}

    Which means $\Phi$ is of type $\omega$ as long as $\omega^k\geq \frac{(x_1-x_0)^{\alpha k}}{\Gamma(\alpha k+1)}\Rightarrow \omega\geq\frac{(x_1-x_0)^{\alpha}}{\Gamma^{\frac{1}{k}}(\alpha k+1)}$ for all integer $k\geq 1$.

    The denominator is not constant, and we wish to minimise it to get a lower bound on $\omega$. Consider the log of the function

    $$\log\Gamma^{\frac{1}{k}}(\alpha k+1)=\frac{\log\Gamma(\alpha k+1)}{k}$$

    Define $f(\alpha k)=\log\Gamma(\alpha k+1)$. Then we take the derivative:

    $$\frac{d}{dk}\frac{f(\alpha k)}{k}=\frac{\alpha kf'(\alpha k)-f(\alpha k)}{k^2}$$

    Consider the numerator. Recall that the Gamma function is log-convex, so $f$ is convex and $b>a\Rightarrow f'(b)>f'(a)$. Hence, by Mean Value Theorem

    $$\frac{f(\alpha k)-f(0)}{\alpha k-0}=\frac{f(\alpha k)}{\alpha k}=f'(c)<f'(\alpha k)$$

    Where $c\in(0, \alpha k)$. Therefore, $\alpha kf'(\alpha k)>f(\alpha k)$, so $\frac{f(\alpha k)}{k}$ is an increasing function and minimised at $k=1$. Therefore, as $\log$ is monotone, $\Gamma^{\frac{1}{k}}(\alpha k+1)$ is minimised at $k=1$, where it is equal to $\Gamma(\alpha+1)$. 

    Hence, for $\omega\geq\frac{(x_1-x_0)^{\alpha}}{\Gamma(\alpha+1)}$, we have

    $$\sum^\infty_{k=1}\frac{\omega^k}{\lambda^k}\geq E_\alpha\left(\lambda^{-1}(x_1-x_0)^\alpha\right)-1\geq\|A\|_{\mathscr{L}(L^p(x_0,x_1;X))}$$

    which by \Cref{resolventriangle} implies $\Phi$ is of type $(1, \omega)$.

\end{proof}

\begin{remark}
    Here, we demonstrated a method of computing the type of $T$ in general. 
\end{remark}

We will now introduce another classification of one-parameter semigroups, known as \emph{analytic semigroups}, which we will then apply to determine the continuity and analytic properties of the semigroup generated by the fractional integral.

\newpage

\section{Well-behavedness of the semigroup generated by \texorpdfstring{$_{x_0}J^\alpha_x$}{the alpha-order integral}}

In this section, we discuss the convergence properties of $T(\alpha, t)\coloneqq e^{J^\alpha t}$. We first discuss how the operator varies with $t$.

\begin{definition}\emph{\cite[Definition 9.1]{isem15}}\label{analyticdef}
For $\theta\in\left(0,\frac{\pi}{2}\right]$, consider the sector
$$\Sigma_\theta\coloneqq \{z\in\mathbb{C}\backslash \{0\}\mid |\arg(z)| <\theta\}$$
Then, an operator $T: \Sigma_\theta\cup \{0\}\rightarrow\mathscr{L}(X)$ is an \emph{analytic semigroup} of angle $\theta$ if the following conditions are satisfied:
\begin{enumerate}[label=(\roman*)]
\item $T: \Sigma_\theta\rightarrow \mathscr{L}(X)$ is holomorphic.
\item $\forall z,w\in\Sigma_\theta$, the identities below hold
$$T(z)T(w) = T(z+w)$$
and 
$$T(0) = I$$
\item For all $\theta'\in (0,\theta)$ and $f\in X$ we have
$$\lim_{\begin{smallmatrix} z\to 0 &\\ z\in\Sigma_{\theta'}\end{smallmatrix}}T(z)f = f$$
\item If for all $\theta'\in (0, \theta)$ we find that $$\sup\lim_{z\in\Sigma_{\theta'}}\|T(z)\|<\infty$$ then we say that $T$ is a \emph{bounded} linear semigroup.
\end{enumerate}

The \emph{generator}, $A$, of the analytic semigroup $T$ is defined to be the same generator as in the restriction $T: [0,\infty)\rightarrow \mathscr{L}(X)$.
\end{definition}

In particular, semigroups generated by bounded linear operators $A$ are examples of analytic semigroups.

\begin{theorem} \label{generatoranalytic}
    Let $A$ be a bounded linear operator and define

    $$T(z)\coloneqq e^{zA}=\sum^\infty_{n=0}\frac{z^nA^n}{n!}$$

    Then, $T$ is an analytic semigroup with $\theta=\frac{\pi}{2}$.
\end{theorem}

\begin{proof}
    First, note that

    $$\left\|\sum^\infty_{n=0}\frac{z^nA^n}{n!}\right\|\leq\sum^\infty_{n=0}\left\|\frac{z^nA^n}{n!}\right\|\leq\sum^\infty_{n=0}\frac{(|z|\|A\|)^n}{n!}=e^{|z|\|A\|}$$

    Hence, $\sum^\infty_{n=0}\frac{z^nA^n}{n!}$ converges and $e^{zA}$ is bounded.
    
    We next show that the semigroup property continues to hold for $z, w\in\Sigma_\theta$.

    \begin{align*}
    e^{zA}e^{wA}&=\sum^\infty_{i=0}\frac{z^iA^i}{n!}\sum^\infty_{j=0}\frac{w^jA^j}{j!} \\
    &=\sum^\infty_{n=0}\sum^n_{i=0}\frac{z^iw^{n-i}A^n}{i!(n-i)!} \\
    &=\sum^\infty_{n=0}\sum^n_{i=0}\binom{n}{i}z^iw^{n-i}\frac{A^n}{n!} \\
    &=\sum^\infty_{n=0}\frac{(z+w)^nA^n}{n!} \\
    &=e^{(z+w)A} \\
    \end{align*}

    Next, we show $\lim_{z\to0}\|T(z)-I\|=0$.

    Consider any $\varepsilon>0$. Let $\delta=\min\left(\frac{\varepsilon}{2\|A\|}, \frac{1}{\|A\|}\right)$. Suppose $|z|<\delta$. Then, $|z|\|A\|<1$ and

    \begin{align*}
        \left\|\sum^\infty_{n=0}\frac{z^nA^n}{n!}-I\right\|&=\left\|\sum^\infty_{n=1}\frac{z^nA^n}{n!}\right\| \\
        &\leq\sum^\infty_{n=1}\left\|\frac{z^nA^n}{n!}\right\| \\
        &\leq\sum^\infty_{n=1}\frac{(|z|\|A\|)^n}{n!} \\
        &<\sum^\infty_{n=1}\frac{(|z|\|A\|)^n}{2^{n-1}} \\
        &<\sum^\infty_{n=1}\frac{|z|\|A\|}{2^{n-1}}=2|z|\|A\|<\varepsilon
    \end{align*}

    Hence $\lim_{z\to0}\|T(z)-I\|=0$, so $\lim_{z\to0}\|T(z)f-f\|=0$ for any $f\in\mathscr{L}(X)$ and $\lim_{z\to0}T(z)f=f$.

    Finally, we prove that $T'(z)$ exists for $z\in\Sigma_\theta$, implying $T$ is holomorphic over $\Sigma_\theta$. First, we find by the semigroup property that

    $$\lim_{h\to 0}\frac{T(z+h)-T(z)}{h}=T(z)\lim_{h\to 0}\frac{T(h)-I}{h}$$

    Since $T(z)$ is bounded, it suffices to show that $\lim_{h\to 0}\frac{T(h)-I}{h}$ is convergent. We claim that this limit is equal to $A$.

    Consider any $\varepsilon>0$. Let $\delta=\min\left(\frac{\varepsilon}{\|A\|^2}, \frac{1}{\|A\|}\right)$. Then, if $|h|<\delta$,

    \begin{align*}
        \left\|\frac{\sum^\infty_{n=0}\frac{h^nA^n}{n!}-I}{h}-A\right\|&=\left\|\frac{\sum^\infty_{n=1}\frac{h^nA^n}{n!}}{h}-A\right\| \\
        &=\left\|\sum^\infty_{n=1}\frac{h^{n-1}A^n}{n!}-A\right\| \\
        &=\left\|\sum^\infty_{n=0}\frac{h^nA^{n+1}}{(n+1)!}-A\right\| \\
        &=\left\|\sum^\infty_{n=1}\frac{h^nA^{n+1}}{(n+1)!}\right\| \\
        &\leq\|A\|\sum^\infty_{n=1}\left\|\frac{h^nA^n}{(n+1)!}\right\| \\
        &\leq\|A\|\sum^\infty_{n=1}\frac{(|h|\|A\|)^n}{(n+1)!} \\
        &\leq\|A\|\sum^\infty_{n=1}\frac{|h|\|A\|}{2^n} \\
        &\leq\|A\||h|\|A\|<\varepsilon \\
    \end{align*}

    Hence we conclude that $T(z)$ is holomorphic and $T'(z)=T(z)A$.

\end{proof}

Hence, this ensures that, by extending $t$ to complex numbers, $\Phi(\alpha,t)$ is an analytic semigroup of angle $\frac{\pi}{2}$.

\begin{remark}
    Weaker conditions, such as strong continuity or local Lipschitz continuity for any $\delta<|z|$, easily follow from the fact that a semigroup $T$ is analytic.
\end{remark}

We also note a simple corollary:

\begin{corollary} \label{Kanalytic}
    Let $K$ be a complex number with real part $>0$. Let $S(z)=T(\alpha, Kz)$, where $T$ is an analytic semigroup. Then $S$ too is an analytic semigroup with $\theta=\frac{\pi}{2}$.
\end{corollary}
\begin{proof}
    First, if $\Re(z)>0$, then $\Re(Kz)>0$, so $S(z)$ is well defined. Also, $S(z)S(w)=T(\alpha, Kz)T(\alpha, Kw)=T(\alpha, K(z+w))=S(z+w)$. We have $\lim_{z\to0}S(z)=\lim_{z\to0}T(\alpha, Kz)=T(\alpha, \lim_{z\to0} Kz)=I$ by analyticity of $T$ w.r.t. $t$, and finally we have $S'(z)=K\frac{\partial}{\partial z}T(\alpha, z)=KJ^\alpha T(\alpha, z)$ by chain rule. Hence, $S(z)$ is a semigroup.
\end{proof}


Next, we talk about the properties of $\Phi$ w.r.t. $\alpha$. Because $J^\alpha$ does not form a uniformly continuous semigroup, we are unable to make conditions that are as strong. However, we note the following:

\begin{theorem}\label{alphalocalipschitz}
    For $\alpha>0$ and $1<p\leq\infty$, $\Phi(\alpha, z)$ is locally Lipschitz-continuous. That is, there exist $\delta, M>0$ such that

    $$\alpha_1, \alpha_2\in(\alpha-\delta, \alpha+\delta)\Rightarrow \|\Phi(\alpha_1, z)-\Phi(\alpha_2, z)\|\leq M|\alpha_1-\alpha_2|$$
\end{theorem}

\begin{proof}
    WLOG let $\alpha_1\leq\alpha_2$. The series representation of $\Phi(\alpha_1, z), \Phi(\alpha_2, z)$ converge absolutely in the norm topology. Hence a.e. we can write

    \begin{multline*}
        \|\Phi(\alpha_1, z)f(x)-\Phi(\alpha_2, z)f(x)\|_X \\
        =\left\|\sum^\infty_{n=1}\frac{z^n}{n!}\int^{x}_{x_0}\left[\frac{(x-s)^{n\alpha_1-1}}{\Gamma(n\alpha_1)}-\frac{(x-s)^{n\alpha_2-1}}{\Gamma(n\alpha_2)}\right]f(s)\,ds\right\|_X \\
        \leq\sum^\infty_{n=1}\frac{|z|^n}{n!}\int^{x}_{x_0}\left|\frac{(x-s)^{n\alpha_1-1}}{\Gamma(n\alpha_1)}-\frac{(x-s)^{n\alpha_2-1}}{\Gamma(n\alpha_2)}\right|\|f(s)\|_X\,ds
    \end{multline*}

    Note that if we take $w=x-s$, $$\frac{\partial}{\partial\beta}\frac{(x-s)^{\beta-1}}{\Gamma(\beta)}=\frac{w^{\beta-1}(\log w-\psi(\beta))}{\Gamma(\beta)}$$

    Hence, by mean value inequality,

    \begin{multline*}
        \|\Phi(\alpha_1, z)f(x)-\Phi(\alpha_2, z)f(x)\|_X \\
        \leq\sup_{\alpha'\in(\alpha_1, \alpha_2)}\sum^\infty_{n=1}\frac{|z|^n}{(n-1)!}\int^{x}_{x_0}\frac{(x-s)^{n\alpha'-1}|\log (x-s)-\psi(n\alpha')|}{\Gamma(n\alpha')}|\alpha_2-\alpha_1|f(s)\,ds
    \end{multline*}

    Note $(\alpha_1, \alpha_2)\subset(\alpha-\delta, \alpha+\delta)$.

    Consider the case $x_1-x_0\geq1$. Using \Cref{normfind}, we find for $1\leq p\leq\infty$:

    \begin{multline*}
        \|\Phi(\alpha_1, z)-\Phi(\alpha_2, z)\|_{\mathscr{L}(L^p(x_0,x_1, X))} \\
        \leq\sup_{\alpha'\in(\alpha_1, \alpha_2)}\sum^\infty_{n=1}\frac{|z|^n}{(n-1)!}\int^{x_1-x_0}_0\frac{w^{n\alpha'-1}|\log w-\psi(\alpha')|}{\Gamma(n\alpha')}|\alpha_2-\alpha_1|\,dw \\
    \end{multline*}
    \begin{multline*}
        \|\Phi(\alpha_1, z)-\Phi(\alpha_2, z)\|_{\mathscr{L}(L^p(x_0,x_1, X))} \\
        \leq|\alpha_2-\alpha_1|\sum^\infty_{n=1}\frac{|z|^n}{(n-1)!}\int^1_0\frac{w^{\alpha-\delta-1}(\psi(n(\alpha+\delta))-\log w)}{\Gammamin}\,dw \\
        +|\alpha_2-\alpha_1|\sum^\infty_{n=1}\frac{|z|^n}{(n-1)!}\sup_{\alpha'\in(\alpha_1, \alpha_2)}\int^{x_1-x_0}_1\frac{w^{n(\alpha+\delta)-1}|\log w-\psi(n\alpha')|}{\Gamma(n\alpha')}\,dw 
    \end{multline*}

    Where $\Gammamin$ is defined as in \Cref{gammamin}. Consider the first sum. We can rewrite as

    \begin{align*}
        &\sum^\infty_{n=1}\frac{|z|^n}{\Gammamin(n-1)!}\sup_{\alpha'\in(\alpha_1, \alpha_2)}\left[\int^1_0w^{n\alpha'-1}\psi(n(\alpha+\delta))\,dw-\int^1_0w^{n\alpha'-1}\log w\,dw\right]\\
        =&\sum^\infty_{n=1}\frac{|z|^n}{\Gammamin(n-1)!}\left[\frac{\psi(n(\alpha+\delta))}{\alpha-\delta}+\frac{1}{(\alpha-\delta)^2}\right]\\
    \end{align*}
    Because $\alpha>\delta$, the integrals converge. From \Cref{digammasymp}, $\psi(n((\alpha+\delta)))\sim\ln n+\ln(\alpha+\delta)$. The ratio of consecutive terms is

    $$\frac{\frac{|z|^{n+1}}{\Gammamin n!}\left[\frac{\ln(n+1)+\ln(\alpha+\delta)}{(\alpha-\delta)}+\frac{1}{(\alpha-\delta)^2}\right]}{\frac{|z|^n}{\Gammamin (n-1)!}\left[\frac{\ln n+\ln(\alpha+\delta)}{(\alpha-\delta)}+\frac{1}{(\alpha-\delta)^2}\right]}=\frac{|z|}{n}\frac{(\alpha-\delta)(\ln(n+1)+\ln(\alpha+\delta))+1}{(\alpha-\delta)(\ln n+\ln(\alpha+\delta))+1}$$

    We can clearly see that the left fraction converges to $0$ as $n$ approaches infinity. Furthermore, we find asymptotically $(\alpha-\delta)(\ln(n+1)+\ln(\alpha+\delta))+1\sim(\alpha-\delta)(\ln n+\ln(\alpha+\delta))+1+\frac{\alpha-\delta}{n}$, so the right fraction converges to $1$. Hence, by the ratio test, the first sum converges to a finite number that we shall label $K$.

    Now consider the second sum:

    $$\sum^\infty_{n=1}\frac{|z|^n}{n!}\int^{x_1-x_0}_1\frac{w^{n\alpha'-1}|\log w-\psi(n\alpha')|}{\Gamma(n\alpha')}\,dw$$
    
    For large enough $n$, we know that $\psi(n(\alpha-\delta))>\log(x-s)$. Hence, $|\log w-\psi(n\alpha')|<\psi(n\alpha')$ (and $n(\alpha-\delta)>\xmin$ as defined in \Cref{gammamin}, meaning $\Gamma(n\alpha')$ is increasing) for large enough $n$. Hence, for large enough $n$, the $n$th term is at most

    $$\frac{|z|^n}{n!}\int^{x_1-x_0}_1\frac{w^{n\alpha'-1}\psi(n\alpha')}{\Gamma(n\alpha')}\leq\frac{|z|^n}{n!}\frac{(x_1-x_0)^{n(\alpha+\delta)}\psi(n(\alpha+\delta))}{\Gamma(n(\alpha-\delta))}$$

    The ratio of consecutive terms as $n\to\infty$ is

    \begin{align*}
        &\lim_{n\to\infty}\frac{\frac{|z|^{n+1}(x_1-x_0)^{(n+1)(\alpha+\delta)}\psi((n+1)(\alpha+\delta))}{(n+1)!\Gamma((n+1)(\alpha-\delta)+1)}}{\frac{|z|^n(x_1-x_0)^{n(\alpha+\delta)}\psi(n(\alpha+\delta))}{n!\Gamma(n(\alpha-\delta)+1)}} \\
        =&\lim_{n\to\infty}\frac{|z|(x_1-x_0)^{\alpha+\delta}}{(n+1)}\frac{\Gamma(n(\alpha-\delta)+1)}{\Gamma((n+1)(\alpha-\delta)+1)}\frac{\psi(n(\alpha-\delta)+1)}{\psi((n+1)(\alpha-\delta)+1)} \\
        =&\lim_{n\to\infty}\frac{|z|(x_1-x_0)^{\alpha+\delta}}{(n+1)}(n(\alpha-\delta)+1)^{-(\alpha+\delta)}\cdot1 \\
        =&\,0
    \end{align*}

    Hence, by ratio test, the second sum also converges. Say it converges to $L$.

    We find

    $$\|\Phi(\alpha_1, z)-\Phi(\alpha_2,z)\|<|\alpha_1-\alpha_2|(K+L)=M|\alpha_1-\alpha_2|$$

    If $x_1-x_0\leq1$, we take $K=M$ with the same $\alpha$. 

    Hence, there exists $M, \delta>0$ satisfying the properties, and $\Phi(\alpha, z)$ is locally Lipschitz-continuous.
\end{proof}
\begin{remark}
    We note that the only condition on $\delta$ we have defined is that it is smaller than $\alpha$. We can equivalently say that $\Phi$ is Lipschitz continuous on the interval $(a, b)$ where $0<a<b<\infty$.
\end{remark}

We can also prove strong continuity w.r.t. $\alpha$:
\begin{theorem}\label{alphastrongcont}
    For every $\alpha\geq 0, f\in L^p(x_0,x_1;X)$, we find

    $$\lim_{h\to0}\|\Phi(\alpha+h, z)f-\Phi(\alpha, z)f\|_{L^p(x_0,x_1;X)}=0$$
\end{theorem}
\begin{proof}
    Local Lipschitz continuity for $\alpha>0$ implies strong continuity for $\alpha>0$ (Suppose $M, \delta'$ satisfy the local Lipschitz condition. For every $\varepsilon>0$, just set $\delta=\min\left\{\frac{\varepsilon}{M},\delta'\right\}$). Hence, we only need to check $\alpha=0$, i.e. show that for every $\varepsilon$, there exists $\delta$ such that

    $$\lim_{h\to0^+}\|\Phi(h, z)f-e^zf\|_{L^p(x_0,x_1;X)}=0$$

    Recall that $J^\alpha$ is known to be strongly continuous. Hence, for every $\varepsilon_0$, there exists a $\delta_0$ such that

    $$0<h<\delta_0\Rightarrow \|\Phi(h, z)f-e^zf\|_{L^p(x_0,x_1;X)}<\varepsilon_0$$

    Now let $\varepsilon_0=\frac{\varepsilon e^{-|z|}}{2\|f\|_{L^p(x_0,x_1;X)}}$ and set $\delta$ correspondingly. If $x_1-x_0\geq 1$, then let $k$ be the smallest integer such that 
    
    $$\frac{|z|^k}{k!}<\frac{\varepsilon e^{-|z|(x_1-x_0)}}{2(x_1-x_0)\left(1+\frac{1}{\Gammamin}\right)\|f\|_{L^p(x_0,x_1;X)}}$$

    Otherwise, let $k$ be the smallest integer such that

    $$\frac{|z|^k}{k!}<\frac{\varepsilon e^{-|z|}}{2\left(1+\frac{1}{\Gammamin}\right)\|f\|_{L^p(x_0,x_1;X)}}$$

    (This is possible as $\frac{|z|^k}{k!}$ is a decreasing function that approaches $0$ as $k\to\infty$.) Let $h<\delta=\min\left\{\frac{\delta_0}{k}, \frac{1}{k}\right\}$.

    Then, we have:

    \begin{align*}
        &\|\Phi(h, z)f-e^zf\|_{L^p(x_0,x_1;X)} \\
        \leq&\sum^{k-1}_{n=0}\frac{|z|^n}{n!}\|{_{x_0}J^{nh}_{x}}f-f\|_{L^p(x_0,x_1;X)}+\sum^{\infty}_{n=k}\frac{|z|^n}{n!}\|{_{x_0}J^{nh}_{x}}f-f\|_{L^p(x_0,x_1;X)} \\
        \leq&\sum^{k-1}_{n=0}\frac{|z|^n}{n!}\varepsilon_0+\sum^{\infty}_{n=k}\frac{|z|^n}{n!}\|{_{x_0}J^{nh}_{x}}f\|_{L^p(x_0,x_1;X)}+\sum^{\infty}_{n=k}\frac{|z|^n}{n!}\|f\|_{L^p(x_0,x_1;X)} \\
        \leq&\,e^{|z|}\varepsilon_0+\sum^{\infty}_{n=k}\frac{|z|^n}{n!}\frac{(x_1-x_0)^{hn}}{\Gamma(nh+1)}\|f\|_{L^p(x_0,x_1;X)}+\frac{|z|^k}{k!}\sum^{\infty}_{n=0}\frac{|z|^n}{n!}\|f\|_{L^p(x_0,x_1;X)} \\
        <&\,e^{|z|}\varepsilon_0+\frac{|z|^k(x_1-x_0)^{hk}}{k! \Gammamin}\sum^{\infty}_{n=0}\frac{\left(|z|(x_1-x_0)^h\right)^n}{n!}\|f\|_{L^p(x_0,x_1;X)}+\frac{|z|^k}{k!}\|f\|_{L^p(x_0,x_1;X)}e^{|z|} \\
        =&\,\frac{\varepsilon}{2}+\frac{|z|^k(x_1-x_0)^{hk}}{k!\Gammamin}e^{|z|(x_1-x_0)^h}\|f\|_{L^p(x_0,x_1;X)}+\frac{|z|^k}{k!}\|f\|_{L^p(x_0,x_1;X)}e^{|z|} \\
        =&\,\frac{\varepsilon}{2}+\frac{|z|^k}{k!}\|f\|_{L^p(x_0,x_1;X)}\left(e^{|z|}+\frac{(x_1-x_0)^{hk}e^{|z|(x_1-x_0)^h}}{\Gammamin}\right)
    \end{align*}

    If $x_1-x_0\geq1$, then because $hk<1$,

    \begin{align*}
        &\frac{|z|^k}{k!}\|f\|_{L^p(x_0,x_1;X)}e^{|z|}\left(1+\frac{(x_1-x_0)^{hk}e^{|z|(x_1-x_0)^h}}{\Gamma_m}\right) \\
        \leq&\,\frac{|z|^k}{k!}\|f\|_{L^p(x_0,x_1;X)}\left(e^{|z|}+\frac{(x_1-x_0)e^{|z|(x_1-x_0)}}{\Gamma_m}\right) \\
        \leq&\,\frac{|z|^k}{k!}\|f\|_{L^p(x_0,x_1;X)}e^{|z|(x_1-x_0)}(x_1-x_0)\left(1+\frac{1}{\Gamma_m}\right) \\
        <\frac{\varepsilon}{2}
    \end{align*}

    Otherwise,

    \begin{align*}
        &\frac{|z|^k}{k!}\|f\|_{L^p(x_0,x_1;X)}\left(e^{|z|}+\frac{(x_1-x_0)^{hk}e^{|z|(x_1-x_0)^h}}{\Gamma_m}\right) \\
        &\frac{|z|^k}{k!}\|f\|_{L^p(x_0,x_1;X)}\left(e^{|z|}+\frac{e^{|z|}}{\Gamma_m}\right) \\
        <\frac{\varepsilon}{2}
    \end{align*}

    In either case, we find

    $$\|\Phi(h, z)f-e^zf\|_{L^p(x_0,x_1;X)}<\frac{\varepsilon}{2}+\frac{\varepsilon}{2}=\varepsilon$$

    which concludes the proof.
\end{proof}

We can additionally consider $\Phi$ as taking an input from $(\alpha, z)\in\mathbb{R}_0\times\mathbb{C}_0$ as a vector equipped with the Euclidean norm, and consider properties w.r.t. these. It turns out that we can extend our prior result on strong continuity forward.

\begin{theorem}\label{vectorstrongcont}

    For $f\in L^p(x_0, x_1; X), (\alpha, z), \in\mathbb{R}_0\times\mathbb{C}_0$, for all $\varepsilon>0$, there exists $\delta>0$ such that
    $$(\alpha', z')\in\mathbb{R}_0\times\mathbb{C}_0, 0<\|(\alpha', z')-(\alpha, z)\|<\delta\Rightarrow \|\Phi(\alpha', z')f-\Phi(\alpha, z)f\|_{L^p(x_0,x_1;X)}<\varepsilon$$
\end{theorem}
\begin{proof}
    By local Lipschitz continuity, we know that there exists $0<\delta'_\alpha$ such that if $|\alpha'-\alpha|<\delta'_\alpha$, then

    $$\|\Phi(\alpha', z')f-\Phi(\alpha, z')f\|_{L^p(x_0,x_1;X)}<M|\alpha'-\alpha|$$

    Now set $\delta_\alpha=\min\left\{\frac{\varepsilon}{2M}, \delta'_\alpha\right\}$.

    Furthermore, by the analytic properties of $\Phi$, we find for all $\varepsilon>0$, there exists $\delta_z>0$ such that
    $$\|\Phi(\alpha, z')f-\Phi(\alpha, z)f\|_{L^p(x_0,x_1;X)}<\frac{\varepsilon}{2}$$

    Next, set $\delta=\min\left\{\delta_\alpha, \delta_z\right\}$. We find that $|\alpha'-\alpha|, |z'-z|<\delta$. Then

    \begin{align*}
        &\|\Phi(\alpha', z')f-\Phi(\alpha, z)f\|_{L^p(x_0,x_1;X)}\\
        &\leq\|\Phi(\alpha', z')f-\Phi(\alpha, z')f\|_{L^p(x_0,x_1;X)}+\|\Phi(\alpha, z')f-\Phi(\alpha, z)f\|_{L^p(x_0,x_1;X)} \\
        &<\frac{\varepsilon}{2}+\frac{\varepsilon}{2}=\varepsilon \\
    \end{align*}

    Which concludes the proof.
\end{proof}

\newpage

\section{Discussion}

In this paper, we have defined the semigroup generated by the fractional integral operator $\Phi$ and determined some of its important properties, primarily through the use of the theory of one-parameter semigroups. In section 3, we were able to explicitly calculate a closed form expression of the resolvent of the fractional integral operator, which then paved the way for section 4, where we then determined the exponential bound of $\Phi$, noticing that it is a strongly continuous semigroup of type $(1,\omega)$. We then moved onto section 5, whereby we demonstrated that, by extending $t$ to complex numbers $z$, $\Phi$ is an analytic semigroup of angle $\frac{\pi}{2}$, which allowed us to then show the well-behavedness of $\Phi$ in the form of Lipschitz continuity and strong continuity, with respect to $\alpha$ as well as both $\alpha$ and $z$.

It turns out that there are several applications to our results, of which we will attempt to look into three.

Firstly, since it is clear that integer order integration is better-behaved than fractional order integration, we can utilise our results on the well-behavedness of $\Phi (\alpha, z)$, particularly \Cref{alphalocalipschitz}, to approximate the error produced when approximating $\Phi (\alpha, z)$ as $\Phi (n, z)$, whereby $n$ is the closest integer to $\alpha$.

Through the use of the complex exponential \textit{Fourier series} of a function $f$ in $L^2(x_0,x_1;X)$, which is well known to be convergent:
$$ f(x) = \sum_{k=-\infty}^\infty e^{\frac{2\pi ik(x-x_0)}{x_1-x_0}}u_k , \hspace{5mm} u_k \in X$$
we can see that it is straightforward to determine the result when the approximated operator $\Phi (n,z)$ is applied to $f$ by integrating each term $n$ times.

Another way we can apply our results is through \textit{Cauchy's integral formula} which (rather surprisingly) extends to operators as well, providing an alternative method to computing $f({_{x_0}J^\alpha_{x}})$ for holomorphic functions $f$:
$$f({_{x_0}J^\alpha_{x}}) = \frac{1}{2\pi i}\int_\Gamma f(\zeta)R(\zeta, {_{x_0}J^\alpha_{x}}) d\zeta$$
As we have determined a closed form expression for the resolvent of ${_{x_0}J^\alpha_{x}}$, it this method of computation may be easier than, perhaps, a direct computation of $f({_{x_0}J^\alpha_{x}})$.

Lastly, our results on the analyticity and continuity of $\Phi$ can also be used to show that solutions of PDEs involving ${_{x_0}J^\alpha_{x}}$, namely fractional PDEs and fractional-integro-PDEs, are well-behaved. For example, 

\begin{theorem}\emph{\cite[Theorem 12.44]{renardy2004}}\label{renardyode}
    Consider the inhomogeneous ODE  
    $$\dot{u}(t)=K{_{x_0}J^\alpha_{x_1}}u(t)+f(t), \hspace{5mm} u(0)=u_0$$

    Where $\Re(K)>0$. Then, $\dot{u}(t)$ is $\theta$-Hölder continuous on $(0,T]$, given that $_{x_0}J^\alpha_xu(t)$ is $\theta$-Hölder continuous on $[0,T]$, for any $0<\theta<1$.
\end{theorem}
\begin{remark}
    The above fractional-integro-PDE can be converted into the following fractional PDE\footnote{We convert it into this form because fractional PDEs tend to have more applications in modelling physical phenomena compared to fractional-integro-PDEs.}:
$$D^{1-\alpha}_{\text{RL}}\dot{u} = u(t) + k(t)$$
Where $D^{1-\alpha}_{\text{RL}}$ is the \emph{Riemann-Liouville fractional derivative}, defined as:
$$D^{\alpha}_{\text{RL}}f\coloneqq \frac{d^{\lceil \alpha \rceil}}{dx^{\lceil \alpha \rceil}}{_{x_0}J^{\lceil \alpha \rceil - \alpha}_{x}}f.$$
\end{remark}


\newpage

\section{Appendix}\label{appendix}

\begin{theorem} \label{resolventprops} \emph{\cite{isem15}}
    Let $X$ be a Banach space and A be a \emph{closed linear operator}\footnote{A linear operator $A$ is closed if its domain $\dom{A}$ is complete with respect to the graph norm $\|f\|_A = \|f\| +\|Af\|$, for $f\in \dom{A}$.} with domain $\dom{A}\subseteq X$. Then each of the following hold:
\begin{enumerate}[label=(\roman*)]
    \item \textit{The resolvent set $\rho(A)$ is open, which implies that its complement $\sigma(A)$ is closed.}

    \item For $\lambda\in\rho(A)$, the mapping 
$$\lambda\longmapsto R(\lambda,A)$$ 
is \emph{complex differentiable} and $\forall n\in\mathbb{N}$,
$$\frac{\mathrm{d}^n}{\mathrm{d}\lambda^n}R(\lambda,A) = (-1)^nn!R(\lambda,A)^{n+1}$$
\end{enumerate}
\end{theorem} 

\begin{theorem} \label{integralresolvent} \emph{\cite{isem15}}
    Let $T$ be a uniformly continuous semigroup with the generator $A$. Then $\forall \lambda\in\mathbb{C}$ and $t<0$,
$$e^{-\lambda t}T(t)f - f = (A - \lambda I)\int_0^te^{-\lambda s}T(s)f\hspace{1mm}ds$$
\end{theorem}

\begin{proof}
It can easily be shown that $e^{-\lambda t}T(t)$ is also a uniformly continuous semigroup with the generator $(A-\lambda I)$ as, by \Cref{uniformsemigroup}:
$$e^{-\lambda t}T(t) = e^{-\lambda t}e^{At} = e^{(A-\lambda I)t}$$
Thus, by part (ii) of \Cref{generatorprops},
$$e^{-\lambda t}T(t)f - f = (A - \lambda I)\int_0^te^{-\lambda s}T(s)f\hspace{1mm}ds$$
\end{proof}

\begin{theorem} \label{resolventorderintegral} \emph{\cite{isem15}}
Let $T$ be a uniformly continuous semigroup of type $(M,\omega)$ with the infinitesimal generator A. Then the following hold, given $\Re(\lambda)>\omega$:

\begin{enumerate}[label=(\roman*)]
    \item  $\forall f\in X$ and $\lambda\in\mathbb{C}$,
$$R(\lambda,A)f = \int_0^\infty e^{-\lambda s}T(s)f\hspace{1mm}ds$$

    \item $\forall f\in X$, $\lambda\in\mathbb{C}$ and $n\in\mathbb{N}$,
$$R(\lambda,A)^nf = \frac{1}{(n-1)!}\int_0^\infty s^{n-1}e^{-\lambda s}T(s)f\,ds$$
\end{enumerate}
\end{theorem}
\begin{proof}
From \Cref{integralresolvent}, we have that:
$$e^{-\lambda t}T(t)f - f = (A-\lambda I)\int_0^te^{-\lambda s}T(s)f\hspace{1mm}ds$$
$$\therefore \lim_{t\to\infty}(e^{-\lambda t}T(t)f - f) = (A-\lambda I)\int_0^\infty e^{-\lambda s}T(s)f\hspace{1mm}ds$$
Since it is given that $\lambda>\omega$, the first term in the limit tends to 0 as $t\to\infty$ and we have,
$$-f = (A-\lambda I)\int_0^\infty e^{-\lambda s}T(s)f\hspace{1mm}ds$$
$$\therefore f = (\lambda I - A)\int_0^\infty e^{-\lambda s}T(s)f\hspace{1mm}ds$$
\begin{flalign*}
\Rightarrow R(\lambda,A)f &= (\lambda-A)^{-1}(\lambda-A)\int_0^\infty e^{-\lambda s}T(s)f\hspace{1mm}ds\\
                          &= \int_0^\infty e^{-\lambda s}T(s)f\hspace{1mm}ds
\end{flalign*}
Thus, (i) is proved. Now, notice that, by rearranging the identity in part (ii) of \Cref{resolventprops} and by \Cref{leibnizintrule}, as well as part (i) of this theorem, we have
\begin{flalign*}
R(\lambda,A)^nf &= \frac{(-1)^{n-1}}{(n-1)!}\frac{\mathrm{d}^{n-1}}{\mathrm{d}\lambda^{n-1}}R(\lambda,A)f\\
                &=\frac{(-1)^{n-1}}{(n-1)!}\int_0^\infty\frac{\partial^{n-1}}{\partial\lambda^{n-1}}(e^{-\lambda s}T(s)f)\hspace{1mm}ds\\
                &=\frac{(-1)^{n-1}}{(n-1)!}\int_0^\infty(-1)^{n-1}s^{n-1}e^{-\lambda s}T(s)f\hspace{1mm}ds\\
                &=\frac{1}{(n-1)!}\int_0^\infty s^{n-1}e^{-\lambda s}T(s)f\hspace{1mm}ds
\end{flalign*}
Hence, we have shown that part (ii) holds. 
\end{proof}

\begin{theorem} \label{hilleyosidaproof} \emph{\cite{isem15, rudin1991}} \emph{(Hille-Yosida Theorem)} Let $T:[0, \infty)\to\mathscr{L}(X)$ be a strongly continuous semigroup with generator $A$ whose domain is dense in $X$. Then,

$$\|T(t)\|\leq Me^{\omega t} \iff \|R(\lambda, A)^n\|\leq\frac{M}{(\lambda-\omega)^n}$$
\end{theorem}
\begin{proof}
We first show the forward direction. We can apply part (ii) of \Cref{resolventorderintegral} to form the following inequality:
\begin{flalign*}
\|R(\lambda,A)^nf\| &\leq \frac{1}{(n-1)!}\int_0^\infty s^{n-1}e^{-\lambda s}Me^{\omega s}\|f\|\hspace{1mm}ds\\
                    &\leq \frac{M\|f\|}{(n-1)!}\int_0^\infty s^{n-1}e^{(\omega-\lambda) s}\hspace{1mm}ds
\end{flalign*}
Now notice that the integral on right hand side of the last inequality is actually the \emph{Laplace transform} of $s^{n-1}$, given by
$$\mathcal{L}\{s^{n-1}\}(\lambda-\omega) = \frac{(n-1)!}{(\lambda-\omega)^n}$$
By substituting this into our inequality, we obtain,
$$\|R(\lambda,A)^nf\|\leq \frac{M}{(\lambda - \omega)^n}\|f\|$$
$$\Rightarrow \|R(\lambda,A)^n\|\leq \frac{M}{(\lambda - \omega)^n}$$
which concludes the forward direction. We now prove the backward direction.

Let $S(\varepsilon)=(1-\varepsilon A)^{-1}$. Our condition becomes

$$\|S(\varepsilon)^n\|\leq M(I-\varepsilon \omega)^{-n}, 0<\varepsilon<\varepsilon_0=\frac{1}{\omega}$$

Furthermore, $S(\varepsilon)(I-\varepsilon A)f=f$ when $f\in \dom{A}$. Expanding out, this tells us $f-S(\varepsilon)f=-\varepsilon S(\varepsilon)Af$. As $\varepsilon\to 0$, the norm of the RHS approaches $0$, i.e.

$$\lim_{\varepsilon\to0}S(\varepsilon) f=f$$

Because $\dom{A}$ is dense in $X$, for any $f$ we can find a sequence $(f_n)$ where $f_n\to f$ and each $f_n\in \dom{A}$. Because $S(\varepsilon)-I$ has norm bounded by $M(1-\varepsilon_0 A)^{-1}+1$, and $\lim_{\varepsilon\to0}S(\varepsilon)f_n=f_n$, we conclude 
\begin{align*}
    \lim_{\varepsilon\to0}\|S(\varepsilon)f-f\|&\leq\lim_{\varepsilon\to0}\|(S(\varepsilon)-I)f_n\|-\|(S(\varepsilon)-I)(f-f_n)\| \\
    &\leq0+\left(1+M(1-\varepsilon_0\omega)^{-1}\right)(f-f_n) \\
    &\stackrel{f_n\to f}{=}0
\end{align*} for all $f\in X$.

Now let

$$T(t,\varepsilon)=e^{tAS(\varepsilon)}$$

Given that $\varepsilon AS(\varepsilon)=S(\varepsilon)-I$, we know

$$T(t,\varepsilon)=\exp\left(t\frac{S(\varepsilon)-I}{\varepsilon}\right)=\exp\left(-\frac{t}{\varepsilon}\right)\sum^\infty_{n=0}\frac{t^n}{n!\varepsilon^n}S(\varepsilon)^n$$

Then,

\begin{align*}
    \|T(t,\varepsilon)\|&\leq \exp\left(-\frac{t}{\varepsilon}\right)\sum^\infty_{n=0}\frac{t^n}{n!\varepsilon^n}\|S(\varepsilon)^n\| \\
    &=\exp\left(-\frac{t}{\varepsilon}\right)\sum^\infty_{n=0}\frac{M}{n!}\left(\frac{t}{\varepsilon(1-\varepsilon\omega)}\right)^n\\
    &=M\exp\left(\frac{t}{\varepsilon}\left(\frac{1}{1-\varepsilon\omega}-1\right)\right)\\
    &=M\exp\left(\frac{t}{\varepsilon}\frac{\varepsilon\omega}{1-\varepsilon\omega}\right)\\
    &=M\exp\left(\frac{t\omega}{1-\varepsilon\omega}\right)
\end{align*}

Furthermore, we know $S(\varepsilon)$ and $A$ commute for $f\in\dom{A}$ because $(I-\varepsilon A)S(\varepsilon)f=S(\varepsilon)(I-\varepsilon A)f$. Thus, for $\varepsilon>0$ and $f\in\dom{A}$

\begin{align*}
    \frac{\partial}{\partial \varepsilon} T(t, \varepsilon) f & =\frac{\partial}{\partial \varepsilon}\left[t A(I-\varepsilon A)^{-1}\right] T(t, \varepsilon) f \\
    & =-t A(I-\varepsilon A)^{-2} \cdot-A(t, \varepsilon) f \\
    & =t(A S(\varepsilon))^2 T(t, \varepsilon) f \\
    & =t T(t, \varepsilon) \frac{(S(\varepsilon)-I)^2}{\varepsilon^2} f \\
    \left\|\frac{\partial}{\partial \varepsilon} T(t, \varepsilon) f\right\| & \leq t M \exp \left(\frac{t \omega}{1-\varepsilon \omega}\right) \frac{1}{\varepsilon}\left\|\frac{(S(\varepsilon)-I)^2}{\varepsilon} f\right\|
\end{align*}

By Mean Value Inequality, 

\begin{align*}
    \left\|[T(t,\varepsilon+h)-T(t,\varepsilon)]f\right\|&\leq\sup_{\varepsilon'\in(\varepsilon,\varepsilon+h)}h\left.\left\|\frac{\partial}{\partial\delta}T(t,\delta)f\right\|\right|_{\delta=\varepsilon'} \\
    &\leq tM\exp\left(\frac{t\omega}{1-h\omega}\right)\sup_{\varepsilon'\in(\varepsilon,\varepsilon+h)}\frac{h}{\varepsilon'}\left\|(S(\varepsilon')-I)\frac{S(\varepsilon')-I}{\varepsilon'}f\right\| \\
    &<tM\exp\left(\frac{t\omega}{1-h\omega}\right)\sup_{\varepsilon'\in(\varepsilon,\varepsilon+h)}\left\|(S(\varepsilon')-I)\frac{S(\varepsilon')-I}{\varepsilon'}f\right\| \\
\end{align*}

First, take $\varepsilon\to0\Rightarrow\varepsilon'\in(0, h)$. Next, take $h\to0$.

\begin{align*}
    &\lim_{h\to0}\sup_{\varepsilon'\in(0,h)}\left\|(S(\varepsilon')-I)\frac{S(\varepsilon')-I}{\varepsilon'}f\right\| \\
    =&\left\|\limsup_{\varepsilon'\to0}(S(\varepsilon')-I)\frac{\varepsilon'AS(\varepsilon')}{\varepsilon'}f\right\| \\
    =&\left\|\limsup_{\varepsilon'\to0}(S(\varepsilon')-I)S(\varepsilon')Af\right\|=0 \\
\end{align*}

Hence the above limit tends to $0$ and $Q(t)\coloneqq\lim_{\varepsilon\to0}T(t,\varepsilon)f$ is well-defined for $f\in\dom{A}$.

Next, note

$$\|Q(t)\|=\lim_{\varepsilon\to0}\|T(t,\varepsilon)\|=\lim_{\varepsilon\to0}\exp\left(\frac{t\omega}{1-\varepsilon\omega}\right)=\exp(t\omega)$$

Hence $Q(t)$ is bounded. For a sequence $(f_n)\to f$ with each individual term in $\dom{A}$, $(Af_n)$ is a Cauchy sequence. Hence, we can extend the definition of $Q(t)$ to all $f\in X$.

Furthermore, because of the additive and topological properties of $T(t,\varepsilon)$, we can also conclude that $Q(t)$ is a semigroup.

Finally, we show $Q(t)$ is a semigroup generated by $A$.

Let $\widetilde{A}$ be the generator of $Q$. By part (i) of \Cref{resolventorderintegral},

$$(\lambda I-\widetilde{A})^{-1}f=\int^\infty_0 e^{-\lambda s}Q(s)f\,ds$$

However, $AS(\varepsilon)$ generates $T(t,\varepsilon)$ so we also have

$$(\lambda I-AS(\varepsilon))^{-1}f=\int^\infty_0 e^{-\lambda s}T(s,\varepsilon)f\,ds$$

Taking $\varepsilon\to 0$, because of the strong convergence of $S(\varepsilon)\to I$, we conclude

$$(\lambda I-A)^{-1}f=\int^\infty_0 e^{-\lambda s}Q(s)f\,ds$$

Hence $A=\widetilde{A}$, completing the proof.

\end{proof}

\newpage

\section{Acknowledgements}
We would like to express our appreciation to Mr. James Coe for his help in providing direction in our research, and Mr. Thomas Johnson for his help in acting as supervisor to our project.

\printbibliography

\href{https://www.wolframalpha.com/}{Wolfram Alpha} and \href{https://www.integral-calculator.com/}{Integral Calculator} were used in the calculations of various integrals and derivatives.

\end{document}